\documentclass[a4paper]{article}
\usepackage{amssymb}
\usepackage{amsfonts}
\usepackage{amsmath}
\usepackage{amsthm}
\usepackage{authblk}
\usepackage[USenglish]{babel}
\usepackage{esint}
\usepackage{graphicx}
\usepackage[colorlinks=true]{hyperref}
\usepackage[latin1]{inputenc}
\newtheorem{theorem}{Theorem}[section]
\newtheorem{lemma}[theorem]{Lemma}
\newtheorem{proposition}[theorem]{Proposition}
\newtheorem{corollary}[theorem]{Corollary}
\newtheorem{remark}[theorem]{Remark}
\newtheorem{definition}[theorem]{Definition}

\numberwithin{equation}{section}

\newcommand{\bthm}{\begin{theorem}}
\newcommand{\ethm}{\end{theorem}}
\newcommand{\blem}{\begin{lemma}}
\newcommand{\elem}{\end{lemma}}
\newcommand{\bprop}{\begin{proposition}}
\newcommand{\eprop}{\end{proposition}}
\newcommand{\bcor}{\begin{corollary}}
\newcommand{\ecor}{\end{corollary}}
\newcommand{\brem}{\begin{remark}}
\newcommand{\erem}{\end{remark}}
\newcommand{\bdefi}{\begin{definition}}
\newcommand{\edefi}{\end{definition}}
\newcommand{\bpf}{\begin{proof}}
\newcommand{\epf}{\end{proof}}
\newcommand{\bl}{\begin{array}{l}}
\newcommand{\bll}{\begin{array}{ll}}
\newcommand{\barr}{\begin{array}}
\newcommand{\earr}{\end{array}}
\newcommand{\bite}{\begin{itemize}}
\newcommand{\eite}{\end{itemize}}
\newcommand{\bequ}{\begin{equation}}
\newcommand{\eequ}{\end{equation}}
\newcommand{\beqa}{\begin{eqnarray}}
\newcommand{\eeqa}{\end{eqnarray}}
\newcommand{\beqy}{\begin{eqnarray*}}
\newcommand{\eeqy}{\end{eqnarray*}}

\allowdisplaybreaks

\begin{document}

\everymath{\displaystyle}

\title{Ground state solutions for a nonlinear Choquard equation}
\author{Luca Battaglia\thanks{Universit\`a degli studi Roma Tre, Dipartimento di Matematica e Fisica, Largo S. Leonardo Murialdo $1$, $00146$ Roma - lbattaglia@mat.uniroma3.it}}
\date{}

\maketitle

\begin{abstract}${}$\\
We discuss the existence of ground state solutions for the Choquard equation
$$-\Delta u+u=(I_\alpha*F(u))F'(u)\quad\quad\quad\text{in }\mathbb R^N.$$
We prove the existence of solutions under general hypotheses, investigating in particular the case of a homogeneous nonlinearity $F(u)=\frac{|u|^p}p$. The cases $N=2$ and $N\ge3$ are treated differently in some steps. The solutions are found through a variational mountain-pass strategy.\\
The results presented are contained in the papers \cite{mv15,bv}.
\end{abstract}\

\section{Introduction}\

We investigate the existence of solutions for nonlinear Choquard equations of the form
\begin{equation}\label{choq}
-\Delta u+u=(I_\alpha*F(u))F'(u)\quad\quad\quad\text{in }\mathbb R^N,
\end{equation}
where $\Delta$ is the standard Euclidean laplacian, $*$ indicates the convolution, $F\in C^1(\mathbb R,\mathbb R)$ is a smooth nonlinearity and $I_\alpha:\mathbb R^N\to\mathbb R$ is, for $\alpha\in(0,N)$, the Riesz potential:
\begin{equation}\label{riesz}
I_\alpha(x):=\frac{\Gamma\left(\frac{N-\alpha}2\right)}{\Gamma\left(\frac{\alpha}2\right)\pi^\frac{N}22^\alpha}\frac{1}{|x|^{N-\alpha}}.
\end{equation}

Problem \eqref{choq} can be seen as a non-local counterpart of the very well-known scalar field equation
\begin{equation}\label{scal}
-\Delta u+u=G'(u)\quad\quad\quad\text{in }\mathbb R^N,
\end{equation}
which can be formally recovered from \eqref{choq} by letting $\alpha$ go to $0$ and setting $G=\frac{F^2}2$.\\
Problem \eqref{scal} has been widely studied since many years. General existence results were provided in \cite{bl} when $N\ge3$ and \cite{bgk} (when $N=2$) under mild hypotheses on $G$.\\
Anyway, the argument from both \cite{bl} and \cite{bgk} does not seem to be suitable to attack problem \eqref{scal}: roughly speaking, the authors use a constrained minimization technique and then a dilation to get rid of the Lagrangian multiplier, which does not work in our case because of the scaling properties of the Riesz potential \eqref{riesz}.

We study the problem \eqref{choq} variationally: its solutions are critical points of the following energy functional on $H^1\left(\mathbb R^N\right)$:
\begin{equation}\label{i}
\mathcal I(u)=\frac{1}2\int_{\mathbb R^N}\left(|\nabla u|^2+|u|^2\right)-\frac{1}2\int_{\mathbb R^N}(I_\alpha*F(u))F'(u).
\end{equation}
In particular, we look for solutions at a mountain-pass level $b$ defined by
\begin{equation}\label{b}
b:=\inf_{\gamma\in\Gamma}\sup_{t\in[0,1]}\mathcal I(\gamma(t)),
\end{equation}
with
$$\Gamma:=\left\{\gamma\in C\left([0,1],H^1\left(\mathbb R^N\right)\right);\,\gamma(0)=0,\,\mathcal I(\gamma(1))<0\right\}.$$
In particular, we by-pass the issue of Palais-Smale sequences by a \emph{scaling trick} introduced in \cite{jea}, which basically allows us to consider Palais-Smale sequences also asymptotically satisfying the Poho\v zaev identity
\begin{equation}\label{poho}
\mathcal P(u):=\frac{N-2}2\int_{\mathbb R^N}|\nabla u|^2+\frac{N}2|u|^2-\frac{N+\alpha}2\int_{\mathbb R^N}(I_\alpha*F(u))F(u)=0,
\end{equation}
for which convergence is easier to be proved.\\
We can show existence of solutions under general hypotheses, in the same spirit of \cite{bl,bgk}. In the particular yet very important case of a power-type nonlinearity $F(u)=\frac{|u|^p}p$ such hypotheses are equivalent to $1+\frac{\alpha}N<p<\frac{N+\alpha}{N-2}$, which in \cite{mv13} is shown to be also a necessary condition. This shows that the hypotheses we make are somehow natural.\\
We also show that the mountain-pass type solution is also a ground state, namely an energy-minimizing solution: it satisfies
\begin{equation}\label{c}
\mathcal I(u)=c:=\inf\left\{\mathcal I(v):\,v\in H^1\left(\mathbb R^N\right)\setminus\{0\}\text{ solves }\eqref{choq}\right\}.
\end{equation}
We first show the existence of mountain-pass solutions in Section $2$ and then in Section $3$ we prove that they are actually ground states. Such results were originally presented in \cite{mv15} for the dimension $N\ge3$ and in \cite{bv} for the case $N=2$.

\section{Existence of mountain-pass solutions}\

We show here existence of a solution for \eqref{choq} under general hypotheses on $F$.\\
First of all, we want to exclude the trivial case of an identically vanishing $F$:
\begin{itemize}
\item[$(F_0)$] There exists $s_0\in\mathbb R$ such that $F(s_0)\ne0$.
\end{itemize}
Then, we also need some growth assumptions which give a well-posed variational formulation, namely a energy functional $\mathcal I$ being well-defined on $H^1\left(\mathbb R^N\right)$. Such assumptions are different depending whether the dimension is two or it is greater, since the limiting-case embeddings in Sobolev spaces are different: in the higher-dimensional case, we impose a power-type growth whereas in $\mathbb R^2$ we require one of exponential type:
\begin{itemize}
\item[$(N\ge3)\quad(F_1)$] There exists $C>0$ such that $|F'(s)|\le C\left(|s|^\frac{\alpha}N+|s|^\frac{\alpha+2}{N-2}\right)$ for any $s>0$
\item[$(N=2)\quad(F_1')$] For any $\theta>0$ there exists $C_\theta>0$ such that $|F'(s)|\le C_\theta\min\left\{1,|s|^\frac{\alpha}2\right\}e^{\theta|s|^2}$ for any $s>0$.
\end{itemize}
It is not hard to see that $(F_1)$, combined with Sobolev and Hardy-Littlewood-Sobolev inequality, implies the finiteness of the term $\int_{\mathbb R^N}(I_\alpha*F(u))F(u)$, hence the well-posedness and smoothness of the functional $\mathcal I$ defined by \eqref{i}. In dimension two we need, in place of Sobolev's inequality, a special form of the Moser-Trudinger inequality on the whole plane, which was given in \cite{at}:
\begin{equation}\label{mt}
\forall\beta\in(0,4\pi)\exists C_\beta>0\text{ such that}\,\int_{\mathbb R^2}|\nabla u|^2\le1\,\Rightarrow\,\int_{\mathbb R^2}\min\left\{1,u^2\right\}e^{\beta u^2}\le C_\beta\int_{\mathbb R^2}|u|^2
\end{equation}
The last hypotheses we need is a sort of \emph{sub-criticality} with respect to the critical power in Hardy-Littlewood-Sobolev inequality. Again, we state the condition differently depending on the dimension, since in dimension $2$ there is no critical Sobolev exponent:
\begin{itemize}
\item[$(N\ge3)\quad(F_2)$] $\lim_{s\to0}\frac{F(s)}{|s|^{1+\frac{\alpha}N}}=\lim_{s\to+\infty}\frac{F(s)}{|s|^\frac{N+\alpha}{N-2}}=0$
\item[$(N=2)\quad(F_2')$] $\lim_{s\to0}\frac{F(s)}{|s|^{1+\frac{\alpha}N}}=0$
\end{itemize}
Precisely, the result we present is the following:

\begin{theorem}\label{mp}${}$\\
Assume $F$ satisfies $(F_0),(F_1),(F_2)$ if $N\ge3$ and $(F_0),(F'_1),(F'_2)$ if $N=2$. Then, the problem \eqref{choq} has a non-trivial solution $u\in H^1\left(\mathbb R^N\right)\setminus\{0\}$.
\end{theorem}\

We start by showing the existence of a Poho\v zaev-Palais-Smale sequence. We argue as in \cite{jea} to get the asymptotical Poho\v zaev identity.

\begin{lemma}\label{constr}${}$\\
Assume $F$ satisfies $(F_0),(F_1)$ (or, in case $N=2$, $(F_0),(F_1')$). Then, there exists a sequence $(u_n)_{n\in\mathbb N}$ in $H^1\left(\mathbb R^N\right)$ such that:
$$\mathcal I(u_n)\underset{n\to+\infty}\to b\quad\quad\quad\quad\mathcal I'(u_n)\underset{n\to+\infty}\to0\text{ in }{H^1\left(\mathbb R^N\right)}'\quad\quad\quad\quad\mathcal P(u_n)\underset{n\to+\infty}\to0$$
\end{lemma}\

\begin{proof}${}$\\
We divide the proof in three steps: first we show that the mountain-pass level \eqref{b} is not degenerate and then we apply a variant of the mountain-pass principle.
\begin{itemize}
\item[Step 1: $b>0$] We suffice to show that $\Gamma\ne\emptyset$, namely that there exists some $u_0\in H^1\left(\mathbb R^N\right)$ with $\mathcal I(u_0)<0$.\\
By $(F_0)$, we can choose $s_0$ such that $F(s_0)\ne0$, therefore if we take a smooth $v_0$ approximating $s_0\mathbf1_{B_1}$ we easily get $\int_{\mathbb R^N}(I_\alpha*F(v_0))F(v_0)>0$. If now we consider $v_t=v_0\left(\frac{\cdot}t\right)$, we get
\begin{equation}\label{scale}
\mathcal I(v_t)=\frac{t^{N-2}}2\int_{\mathbb R^N}|\nabla v_0|^2+\frac{t^N}2\int_{\mathbb R^N}|v_0|^2-\frac{t^{N+\alpha}}2\int_{\mathbb R^N}(I_\alpha*F(v_0))F(v_0),
\end{equation}
which is negative for large $t$, so we can take $u_0=v_t$ with $t\gg1$.
\item[Step 2: $b<+\infty$] We need to show that for any $\gamma\in\Gamma$ there exists $t_\gamma$ such that $\mathcal I(\gamma(t_\gamma))\ge\varepsilon>0$.\\
If $\int_{\mathbb R^N}\left(|\nabla u|^2+|u|^2\right)\le\delta\ll1$, then by assumption $(F_2)$ and H-L-S and Sobolev's inequality we get
$$\int_{\mathbb R^N}(I_\alpha*F(u))F(u)\le C\left(\left(\int_{\mathbb R^N}|\nabla u|^2\right)^\frac{N+\alpha}{N-2}+\left(\int_{\mathbb R^N}|u|^2\right)^{1+\frac{\alpha}N}\right)\le\frac{1}4\int_{\mathbb R^N}\left(|\nabla u|^2+|u|^2\right),$$
which means $\mathcal I(u)\ge\frac{1}4\int_{\mathbb R^N}\left(|\nabla u|^2+|u|^2\right)$, and the same can be proved similarly when $N=2$.\\
Now, for any fixed $\gamma\in\Gamma$ we can take $t_\gamma$ such that $\int_{\mathbb R^2}\left(|\nabla\gamma(t_\gamma)|^2+|\gamma(t_\gamma)|^2\right)=\delta$ and we get $\mathcal I(\gamma(t_\gamma))\ge\frac{\delta}4=:\varepsilon$.
\item[Step 3: Conclusion] Consider the functional $\widetilde{\mathcal I}:\mathbb R\times H^1\left(\mathbb R^N\right)\to\mathbb R$ defined by $$\widetilde{\mathcal I}(\sigma,v):=\mathcal I\left(v\left(e^{-\sigma}\cdot\right)\right)=\frac{e^{(N-2)\sigma}}2\int_{\mathbb R^N}|\nabla v|^2+\frac{e^{N\sigma}}2\int_{\mathbb R^N}|v|^2-\frac{e^{(N+\alpha)\sigma}}2\int_{\mathbb R^N}(I_\alpha*F(v))F(v).$$
By applying to $\widetilde{\mathcal I}$ the standard min-max principle (see \cite{wil} for instance) we get a sequence $(\sigma_n,v_n)_{n\in\mathbb N}$ with $\widetilde{\mathcal I}(\sigma_n,v_n)\underset{n\to+\infty}\to b$ and ${\widetilde{\mathcal I}(\sigma_n,v_n)}'\underset{n\to+\infty}\to0$, which is equivalent to what the Lemma required.
\end{itemize}
\end{proof}\

To prove Theorem \ref{mp} we need to show the convergence of the Poho\v zaev-Palais-Smale sequence we just found. Here we need the sub-criticality assumption $(F_2),(F_2')$
\begin{lemma}\label{conv}${}$\\
Assume $F$ satisfies $(F_1),(F_2)$ (or, in case $N=2$, $(F_1'),(F_2')$) and $(u_n)_{n\in\mathbb N}$ satisfies
$$\mathcal I(u_n)\text{ is bounded}\quad\quad\quad\quad\mathcal I'(u_n)\underset{n\to+\infty}\to0\text{ in }{H^1\left(\mathbb R^N\right)}'\quad\quad\quad\quad\mathcal P(u_n)\underset{n\to+\infty}\to0.$$
Then, up to subsequences,
\begin{itemize}
\item either $u_n\underset{n\to+\infty}\to0$ strongly in $H^1\left(\mathbb R^N\right)$
\item or $u_n(\cdot-x_n)\underset{n\to+\infty}\rightharpoonup u$ weakly for some $(x_n)_{n\in\mathbb N}$ in $\mathbb R^N$ and $u\in H^1\left(\mathbb R^N\right)\setminus\{0\}$.
\end{itemize}
\end{lemma}\

\begin{proof}${}$\\
Assume the first alternative does not occur. Then, we show it weakly converges to some $u\not\equiv0$.\\
\begin{itemize}
\item[Step 1: $(u_n)_{n\in\mathbb N}$ is bounded] It follows by just writing
$$\frac{\alpha+2}{2(N+\alpha)}\int_{\mathbb R^N}|\nabla u_n|^2+\frac{\alpha}{2(N+\alpha)}\int_{\mathbb R^N}|u_n|^2=\mathcal I(u_n)-\frac{\mathcal P(u_n)}{N+\alpha}\underset{n\to+\infty}\to b$$
\item[Step 2: $\sup_{x\in\mathbb R^N}\int_{B_1(x)}|u_n|^p\ge\frac{1}C$] By using the asymptotic Poho\v zaev identity it is not hard to see that $\inf_n\int_{\mathbb R^N}(I_\alpha*F(u_n))F(u_n)>0$. Moreover, $(F_2)$ implies, for any $\epsilon>0,p\in\left(2,\frac{2N}{N-2}\right)$,
$$|F(s)|^\frac{2N}{N+\alpha}\le\varepsilon\left(|s|^2+|s|^\frac{2N}{N-2}\right)+C_\varepsilon|s|^p,$$
therefore, by the following inequality from \cite{lio}
$$\int_{\mathbb R^N}|u_n|^p\le C\left(\int_{\mathbb R^N}\left(|\nabla u_n|^2+|u_n|^2\right)\right)\left(\sup_{x\in\mathbb R^N}\int_{B_1(x)}|u_n|^p\right)^{1-\frac{2}p},$$
we get
\begin{eqnarray*}
&&\left(\sup_{x\in\mathbb R^N}\int_{B_1(x)}|u_n|^p\right)^{1-\frac{2}p}\ge\frac{1}C\frac{\int_{\mathbb R^N}|u_n|^p}{\int_{\mathbb R^N}\left(|\nabla u_n|^2+|u_n|^2\right)}\\
&\ge&\frac{1}{C_\varepsilon}\left(\int_{\mathbb R^N}|F(u_n)|^\frac{2N}{N+\alpha}-\varepsilon\int_{\mathbb R^N}\left(|u|^2+|u|^\frac{2N}{N-2}\right)\right)\\
&\ge&\frac{1}{C'_\varepsilon}\left(\left(\int_{\mathbb R^N}(I_\alpha*F(u_n))F(u_n)\right)^\frac{N}{N+\alpha}-C\varepsilon\int_{\mathbb R^N}\left(|\nabla u_n|^2+|u_n|^2\right)\right)\ge\frac{1}C.
\end{eqnarray*}
and a similar estimate holds true in the case $N=2$.
\item[Step 3: $u_n(\cdot-x_n)$ converges] We choose $x_n$ such that $\liminf_{n\to+\infty}\int_{B_1}|u_n(\cdot-x_n)|^p>0$, its weak limit (which exists because Step $1$ ensures boundedness) must be some $u\not\equiv0$.\\
By Sobolev embeddings, one can show that $(I_\alpha*F(u_n))F'(u_n)\underset{n\to+\infty}\rightharpoonup (I_\alpha*F(u))F'(u)$ in $L^p_{\mathrm{loc}}\left(\mathbb R^N\right)$. This easily yields that $u$ solves \eqref{choq}
\end{itemize}
\end{proof}\

\begin{proof}[Proof of Theorem \ref{mp}]${}$\\
By Lemma \ref{constr}, $\mathcal I$ admits a Poho\v zaev-Palais-Smale sequence $(u_n)_{n\in\mathbb N}$ at the energy level $b$. We apply Lemma \ref{conv} to the latter sequence: if the first alternative occurred, then we would have $\mathcal I(u_n)\underset{n\to+\infty}\to\mathcal I(0)=0$, contradicting Lemma \ref{conv}. Therefore, the second alternative must occur and in particular $u\not\equiv0$ solves \eqref{choq}.
\end{proof}\

We conclude this section by showing that Theorem \ref{mp} is actually sharp in the case of a power nonlinearity $F(u)=\frac{|u|^p}p$; in other words, we give a non-existence result for all the values $p$ not matching the assumptions of Theorem \ref{mp}. To show non-existence, we use a Poho\v zaev identity, which is a classical property of solutions of \eqref{choq}.

\begin{proposition}${}$\\
Any solution $u$ of \eqref{choq} satisfies the Poho\v zaev identity \eqref{poho}.
\end{proposition}\

\begin{theorem}${}$\\
If $F(u)=\frac{|u|^p}p$ then problem \eqref{choq} admits a non-trivial solution if and only if $p\in\left(1+\frac{\alpha}N,\frac{N+\alpha}{N-2}\right)$, with the latter condition to be read as $p>1+\frac{\alpha}2$ if $N=2$.
\end{theorem}\

\begin{proof}${}$\\
If $p\in\left(1+\frac{\alpha}N,\frac{N+\alpha}{N-2}\right)$ then one can easily see that $F(u)=\frac{|u|^p}p$ satisfies $(F_0),(F_1),(F_2)$, hence the existence of non-trivial solutions follows from Theorem \ref{mp}.\\
Conversely, assume $p$ is outside that range and $u$ solves \eqref{choq}, By testing both sides of against $u$ we get
$$\int_{\mathbb R^N}\left(|\nabla u|^2+|u|^2\right)=\int_{\mathbb R^N}(I_\alpha*|u|^p)|u|^p.$$
Moreover, $u$ satisfies the Poho\v zaev identity \eqref{poho}, which has the form
$$\frac{N-2}2\int_{\mathbb R^N}|\nabla u|^2+\frac{N}2\int_{\mathbb R^N}|u|^2-\frac{N+\alpha}{2p}\int_{\mathbb R^N}(I_\alpha*|u|^p)|u|^p=0.$$
A linear combination of the two formulas gives
$$\left(\frac{N-2}2-\frac{N+\alpha}{2p}\right)\int_{\mathbb R^N}|\nabla u|^2+\left(\frac{N}2-\frac{N+\alpha}{2p}\right)|u|^2,$$
which implies $u\equiv0$ if $p\le1+\frac{\alpha}N$ or $p\ge\frac{N+\alpha}{N-2}$.
\end{proof}\

\section{From solutions to groundstates}

In the last part of this paper we show that the mountain pass solutions given by Theorem \ref{mp} are actually energy-minimizing, in the sense of \eqref{c}.
\begin{theorem}\label{gs}${}$\\
The mountain-pass solution found in Theorem \ref{mp} is actually a ground state, namely its energy level is given by \eqref{c}.
\end{theorem}\

The previous Theorem can be easily proved by constructing, for any solution $v$ of \eqref{choq}, a path $\gamma_v\in\Gamma$ which attains its maximum energy on $v$.
\begin{lemma}\label{path}${}$\\
Assume $F$ satisfies $(F_1)$ and $v\in H^1\left(\mathbb R^N\right)\setminus\{0\}$ solves \eqref{choq}. Then, there exists a path $\gamma_v\in\Gamma$ such that:
$$\gamma(0)=0\quad\quad\quad\gamma_v\left(\frac{1}2\right)=v\quad\quad\quad\mathcal I(\gamma_v(t))<\mathcal I(v)\text{ for any }t\ne\frac{1}2\quad\quad\quad\mathcal I(\gamma_v(t))<0$$
\end{lemma}\

\begin{proof}${}$\\
Fix a non-trivial solution $v$ of \eqref{choq} and consider the path $\overline\gamma_v=:[0,+\infty)\to H^1\left(\mathbb R^N\right)$ defined by $\overline\gamma_v(t)=\left\{\begin{array}{ll}v\left(\frac{\cdot}t\right)&\text{if }t>0\\0&\text{if }t=0\end{array}\right.$.\\
Along the path, the energy is given by \eqref{scale}, which is negative for $t\gg1$. Moreover, due to the Poho\v zaev identity \eqref{poho} we can also write
$$\mathcal I(\overline\gamma_v(t))=\left(\frac{t^{N-2}}2-\frac{N-2}{2(N+\alpha)}t^{N+\alpha}\right)\int_{\mathbb R^N}|\nabla u|^2+\left(\frac{t^N}2-\frac{N}{2(N+\alpha)}t^{N+\alpha}\right)\int_{\mathbb R^N}|u|^2,$$
which has its maximum in $t=1$. Therefore, up to a rescaling of $t$, this path has all the required properties.\\
Anyway, being
$$\int_{\mathbb R^N}\left(|\nabla\overline\gamma_v(t)|^2+|\overline\gamma_v(t)|^2\right)=t^{N-2}\int_{\mathbb R^N}|\nabla v|^2+t^N\int_{\mathbb R^N}|v|^2,$$
$\overline\gamma_v$ is continuous at $t=0$ only if $N\ge3$, so in the case $N=2$ we need a modification for $t$ close to $0$.\\
If $N=2$ we take $\overline\gamma_v(t)=\left\{\begin{array}{ll}v\left(\frac{\cdot}t\right)&\text{if }t>t_0\\\frac{t}{t_0}v\left(\frac{\cdot}{t_0}\right)&\text{if }t\le t_0\end{array}\right.$ for some suitable $t_0\ll1$. We only need to verify that $\mathcal I(\overline\gamma_v(t))\le\mathcal I(\overline\gamma_v(1))$ for $t\le t_0$.\\
Using the assumption $(F_1')$ and Moser-Trudinger's \eqref{mt} and Hardy-Littlewood-Sobolev inequalities we get
$$\int_{\mathbb R^2}(I_\alpha*F(\overline\gamma_v(t)))F(\overline\gamma_v(t))\le C\left(\frac{\int_{\mathbb R^2}|\overline\gamma_v(t)|^2}{\int_{\mathbb R^2}|\nabla\overline\gamma_v(t)|^2}\right)^{1+\frac{\alpha}2}=Ct_0^{2+\alpha}\left(\int_{\mathbb R^2}|v|^2\right)^{1+\frac{\alpha}2},$$
therefore using again Poho\v zaev identity we get, for $t_0$ small enough,
\begin{eqnarray*}
\mathcal I(\overline\gamma_v(t))&\le&\frac{1}2\int_{\mathbb R^2}|\nabla v|^2+\frac{t_0^2}2\int_{\mathbb R^2}|v|^2+Ct_0^{2+\alpha}\left(\int_{\mathbb R^2}|v|^2\right)^{1+\frac{\alpha}2}\\
&=&\mathcal I(v)+\left(\frac{t_0^2}2-\frac{\alpha}{2(2+\alpha)}\right)\int_{\mathbb R^2}|v|^2+Ct_0^{2+\alpha}\left(\int_{\mathbb R^2}|v|^2\right)^{1+\frac{\alpha}2}<\mathcal I(v)
\end{eqnarray*}
and the proof is complete.
\end{proof}\

\begin{proof}[Proof of Theorem \ref{gs}]${}$\\
Let $u$ be the mountain-pass solution found in Theorem \ref{mp}. By the lower-semicontinuity of the norm we find $\mathcal I(u)\le b$, whereas the definition \eqref{c} of ground state yields $\mathcal I(u)\ge c$.\\
Now, take another solution $v\in H^1\left(\mathbb R^N\right)\setminus\{0\}$ and apply Lemma \ref{path}: we get
$$\mathcal I(v)=\sup_{t\in[0,1]}\mathcal I(\gamma_v(t))\ge\inf_{\gamma\in\Gamma}\sup_{t\in[0,1]}\mathcal I(\gamma(t))=b.$$
Being $v$ arbitrary, we get $c\ge b$, hence $c\le\mathcal I(u)\le b\le c$, therefore $\mathcal I(u)=b=c$.
\end{proof}\


\footnotesize
\begin{thebibliography}{99}


%

\bibitem{at}\textsc{Adachi~S., Tanaka~K.}
\textit{Trudinger type inequalities in $\mathbb{R}^N$ and their best exponents},
Proc. Amer. Math. Soc. {\bf 128} (2000), no. 7, 2051-2057

\bibitem{bv}\textsc{Battaglia~L., Van Schaftingen~J.}
\textit{Existence of groundstates for a class of nonlinear Choquard equations in the plane},
Adv. Nonlinear Stud., accepted

\bibitem{bgk}\textsc{Berestycki,~H., Gallou\"et~,T., Kavian,O.}
\textit{\'Equations de champs scalaires euclidiens non lin\'eaires dans le plan},
C. R. Acad. Sci. Paries S\'er. I Math. {\bf 297} (1983), no. 5, 307-310

\bibitem{bl}\textsc{Berestycki,~H., Lions~P.-L.}
\textit{Nonlinear scalar field equations. II. Existence of infinitely many solutions},
Arch. Rational Mech. Anal. {\bf 82} (1983), no. 4, 347-375

\bibitem{jea}\textsc{Jeanjean,~L.}
\textit{Existence of solutions with prescribed norm for semilinear elliptic equations},
Nonlinear Anal. {\bf 28} (1997), no. 10, 1633-1659

\bibitem{lio}\textsc{Lions,~P.-L.},
\textit{The Choquard equation and related questions},
Nonlinear Anal. {\bf 4} (1980), no. 6, 1063-1072

\bibitem{mv13}\textsc{Moroz~V., Van Schaftingen~J.}
\textit{Groundstates of nonlinear Choquard equations: existence, qualitative properties and decay asymptotics},
J. Funct. Anal {\bf 265} (2013), no. 2, 153-184

\bibitem{mv15}\textsc{Moroz~V., Van Schaftingen~J.}
\textit{Existence of groundstates for a class of nonlinear Choquard equations},
Trans. Amer. Math. Soc. {\bf 367} (2015), no. 9, 6557-6579
%
\bibitem{wil}\textsc{Willem,~M.},
{\it Minimax theorems},
Progress in Nonlinear Differential Equations and their Applications, 24, Birkh\"auser, Boston, Mass. 1996

\end{thebibliography}
\end{document}